\documentclass[11pt]{amsart}
\usepackage{graphic x}
\usepackage{amssymb,amsmath}
\usepackage[margin=1in]{geometry}
\usepackage{setspace}
\usepackage{epstopdf}
\usepackage{amsmath,amsthm,amscd,amssymb}
\usepackage{latexsym}
\usepackage[colorlinks,citecolor=red,pagebackref,hypertexnames=false]{hyperref}
\usepackage{geometry}

\numberwithin{equation}{section}
\theoremstyle{plain}
\newtheorem{theorem}{\bf Theorem}
\newtheorem{lemma}{\bf Lemma}[section]

\newtheorem{proposition}{Proposition}
\theoremstyle{definition}

\newtheorem{case[theorem]}{Case}

\theoremstyle{remark}

\numberwithin{equation}{section}

\begin{document}

\title{\parbox{14cm}{\centering{Upper bounds on pairs of dot products}}}
\author{Daniel Barker and Steven Senger}

\subitem \email{poozzab@udel.edu\footnote{Partially supported by a grant from the University of Delaware for summer undergraduate research.}, stevensenger@missouristate.edu}

\thanks{}

\setstretch{1.25}

\begin{abstract} Given a large finite point set, $P\subset \mathbb R^2$, we obtain upper bounds on the number of triples of points that determine a given pair of dot products. That is, for any pair of positive real numbers, $(\alpha, \beta)$,  we bound the size of the set $$\left\{(p,q,r)\in P \times P \times P : p \cdot q = \alpha, p \cdot r = \beta   \right\}.$$

\end{abstract}

\maketitle

\section{Introduction}

Many elementary problems in geometric combinatorics ask how often a particular type of point configuration can occur in subsets of some ambient space. One of the most famous is the Erd\H os single distance problem, which asks how often any fixed distance can occur in a large finite set of points in the plane. The conjecture is that for a set of $n$ points, no distance can occur more than $Cn^{1+\epsilon}$ times, for some constant $C$, independent of $n$, and any $\epsilon >0$. The best known estimate of this is due to Spencer, Szemer\` edi, and Trotter, in \cite{SST}, who have shown that $Cn^\frac{4}{3}$ is an upper bound. A closely related problem, the Erd\H os distinct distances problem, asks for a lower bound on the number of distinct distances determined by point pairs from a large finite point set. This was resolved in the plane by Guth and Katz, in \cite{GK}. Analogous questions have been studied for dot products; see \cite{IRR}, \cite{Steinerberger}, and \cite{GIS}.

Here, we consider triples of points which determine a pair of dot products in large finite point sets. In the settings of vector spaces over various finite rings, there has been activity on the special case of zero dot products by the second listed author, and Iosevich \cite{IS}, as well as Pham and Vinh, in \cite{Vinh}. Information about the dot products determined by a point set finds applications in varied areas such as coding theory, \cite{AB}, graph theory, \cite{Bahls}, and frame theory, \cite{Fickus}. 

We now fix some notation. In what follows, if two quantities, $X$ and $Y$, vary with respect to some natural number parameter, $n$, then we write $X \lesssim Y$ if there exist constants, $C$ and $N$, both independent of $n$, such that for all $n> N$, we have $X\leq CY$. If $X \lesssim Y$ and $Y \lesssim X$, we write $X \approx Y.$  Given a set of points, $P\subset [0,1]^2$, let $\Pi_{\alpha , \beta} (P)$ denote the number of distinct triples of points, that determine a given pair of dot products. That is, for real numbers $\alpha$ and $\beta$,
$$\Pi_{\alpha,\beta}(P) = \{(p,q,r)\in P \times P \times P : p\cdot q = \alpha \text{ and } p \cdot r = \beta\}.$$

We will typically restrict $\alpha$ and $\beta$ to be positive nonzero, because zero dot products behave differently, as demonstrated by Proposition \ref{zeroDPs}. Our first main result applies to any point set in the plane.

\newpage

\begin{theorem} \label{general}
Given a set, $P$, of $n$ points in $\mathbb R^2$, and fixed $\alpha , \beta \neq 0$,
$$|\Pi_{\alpha,\beta}(P)|\lesssim n^2.$$
\end{theorem}

In general, for a set of $n$ points, $P\subset \mathbb R^2$, one cannot expect to get an upper bound better than Theorem \ref{general}, as shown in an explicit construction below, Proposition \ref{sharp}. In cases where the points are evenly-distributed, such as applications with sensor placement or code construction, we have tighter bounds on $\Pi_{\alpha,\beta}(P)$. Our second main result is for point sets with a minimum separation between points.

\begin{theorem}\label{separation}
Let $P \subset [0,1]^2$ be a set of $n$ points that obeys the following separation condition
$$\min\{|p-q| : p,q\in P, ~p\neq q \}\geq \epsilon.$$ For $\epsilon > 0$, and fixed $\alpha , \beta \neq 0$, we have
$$|\Pi_{\alpha , \beta} (P)| \lesssim  n^{\frac{4}{3}}\epsilon^{-1}\log \left(\epsilon^{-1}\right).$$
\end{theorem}

Notice that if $\epsilon$ is chosen to be too small, Theorem \ref{separation} is outdone by Theorem \ref{general}. Similarly, if $\epsilon$ is close to 1, there cannot be many points in the unit square. Keeping this in mind, the range in which Theorem \ref{separation} is most useful is $n^{-\frac{2}{3}}<\epsilon\leq n^{-\frac{1}{2}}$. This range lines up with other results on a wide class of finite point sets called $s$-adaptable sets.

We now introduce the notion of $s$-adaptability. This should be viewed as a measure for how well-distributed the points are. This property has been used to study many types of geometric point configuration problems. See \cite{IJL}, \cite{IRU}, and \cite{IS2}, for example. Families of point sets which are $s$-adaptable can be used to transfer results between discrete point sets and sets with positive Hausdorff dimension.
A large, finite point set $P \subset [0,1]^2$, is said to be {\it $s$-adaptable} if the following two conditions hold:
\begin{align*}
\text{(energy)} \qquad &\frac{1}{{n \choose 2}}\sum_{\substack{p, q \in P\\p\neq q}} |p-q|^{-s} \lesssim 1,\\
\text{(separation)} \qquad &\min\{|p-q| : p,q\in P,~ p\neq q \}\geq n^{-\frac{1}{s}}.\\
\end{align*}
By setting $\epsilon = n^{-\frac{1}{s}}$, and appealing to the definition of $s$-adaptability given here, we get the following estimate as a corollary.
\begin{theorem}\label{main}
Let $P \subset [0,1]^2$ be a set of $n$ points that is s-adaptable. For $2 \geq s > \frac{3}{2}$, and fixed $\alpha , \beta > 0$,
$$|\Pi_{\alpha , \beta} (P)| \lesssim n^{\frac{4}{3} + \frac{1}{s}}\log n.$$
\end{theorem}

First, we construct examples of point sets that illustrating the sharpness of Theorem \ref{general} as well as an illustration of why we assume the restriction of $\alpha,\beta \neq 0$. Next, we prove Theorem \ref{general} and Theorem \ref{separation} in Section \ref{proofs}. Section \ref{lemmaProofs} contains the proofs of two technical lemmas, included for completion.




\section{Explicit constructions}
\subsection{Sharpness of Theorem \ref{general}}
\begin{proposition}\label{sharp}
Given a natural number $n$, and real numbers $0< \alpha,\beta<2$, there is a set, $P$, of $n$ points in $[0,1]^2$ for which
$$|\Pi_{\alpha,\beta}(P)|\approx n^2.$$
\end{proposition}
\begin{proof}
Let $p$ be the point with coordinates $(1,1)$. Now, staying within the unit square, distribute $\left\lfloor \frac{n-1}{2} \right\rfloor$ points along the line $y=\alpha -x$, and distribute the remaining $\left\lceil \frac{n-1}{2} \right\rceil$ points along the line $y=\beta -x$. Clearly, there are $\gtrsim n^2$ pairs of points $(q,r)$, where $q$ is chosen from the first line, and $r$ is chosen from the second. Notice that $p$ contributes a triple to $\Pi_{\alpha,\beta}(P)$ for each such pair, giving us
$$|\Pi_{\alpha,\beta}(P)|\approx n^2.$$
\end{proof}

\subsection{The special case $\alpha=\beta=0$}
\begin{proposition}\label{zeroDPs}
There exists a set, $P$, of $n$ points in $[0,1]^2$ for which
$$|\Pi_{0,0}(P)|\approx n^3.$$
\end{proposition}
\begin{proof}
Arrange $\frac{n}{2}$ along the $x$-axis, and $\frac{n}{2}$ points along the $y$-axis. Now, for each of the $\frac{n}{2}$ points on the $x$-axis, there are $\left(\frac{n}{2}\right)\left(\frac{n}{2}\right)$ pairs of points on the $y$-axis. Notice that any point chosen from the $x$-axis will have dot product zero with each point from the pair chosen from the $y$-axis. Therefore, each of these $\frac{1}{8}n^3$ triples will contribute to $\Pi_{0,0}(P)$.

We can get just as many triples that contribute to $\Pi_{0,0}(P)$ by taking single points from the $y$-axis, and pairs of points from the $x$-axis. In total, we get
$$|\Pi_{0,0}(P)|=\frac{1}{8}n^3+\frac{1}{8}n^3 \approx n^3.$$
\end{proof}

\section{Proofs of main results}\label{proofs}

\subsection{Proof of Theorem \ref{general}}\label{genProof}

\begin{proof}
The basic idea is to estimate the number of triples, $(p,q,r)\in \Pi_{\alpha,\beta}(P)$, by considering pairs of points, $(q,r)\in P\times P$, and bounding the number of possible points, $p\in P$, which could contribute to $\Pi_{\alpha,\beta}(P)$.

Call any line through the origin a {\it radial line}. Let $p\in P$ have coordinates $(p_x,p_y)$. We define the {\it $\alpha$-line} for a point, $p$, to be the set of points that have dot product $\alpha$ with the point $p$. Holding $p$ fixed, this set will be a line of points $q$, with coordinates $(q_x,q_y)$ satisfying the equation
\begin{equation}\label{line}
p\cdot q = p_xq_x+p_yq_y=\alpha.
\end{equation}

\begin{center}
\includegraphics[scale=1]{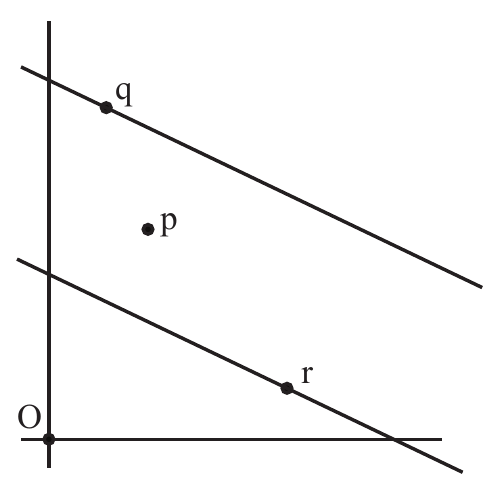}\\
\small{FIGURE 1: Here, $q$ lies on the $\alpha$-line of $p$, and $r$ lies on the $\beta$-line of $p$.}
\end{center}
We can define a {\it $\beta$-line} similarly. By solving \eqref{line} for $q_y$, we can see that the slope of the $\alpha$-line of a point, $p$, will be equal to the slope of the $\beta$-line of the point $p$. Moreover, these lines will be perpendicular to the radial line through $p$. We let $\mathcal L_\alpha (p)$ and $\mathcal L_\beta (p)$ denote the set of all points of $P$ incident to the $\alpha$-line or $\beta$-line for the given point $p$.

Draw $\alpha$-lines and $\beta$-lines for any given pair of $(q,r) \in P \times P$. We will refer to the set of point pairs with distinct $\alpha$-lines and $\beta$-lines as $A$ and pairs with a shared line as $B$. To be precise:
$$A = \{(q,r)\in P\times P:\mathcal L_\alpha (q)\neq \mathcal L_\beta (r) \text{ and }\mathcal L_\beta (q)\neq \mathcal L_\alpha (r)\},$$ $$\text{ and } B = (P\times P) \setminus A.$$

We first consider triples in $\Pi_{\alpha,\beta}(P)$ of the form $(p,q,r)$ where $(q,r)\in A$. Notice that the $\alpha$-lines and $\beta$-lines of a pair of points, $(q,r)\in A$, can intersect at most four times, by definition of the set $A$. So for every pair of points in $A$, there are at most four possible locations for a point in $P$ which would contribute a triple to $\Pi_{\alpha,\beta}(P)$. As $A\subset P\times P$, we see that $|A|\leq n^2$. From this, we see that pairs in $A$ cannot add more than $4n^2$ triples to $\Pi_{\alpha,\beta}(P)$. //\\

We now turn our attention to the set $B$. Without loss of generality, suppose that the $\alpha$-line of a point, $q\in P$, coincides with the $\beta$-line of a point, $r\in P$, then we appeal to the following lemma.

\begin{lemma}\label{setB}
For any pair of points, $(q,r)\in B$, as defined above, the following hold:
\begin{enumerate}
\item  Both $q$ and $r$ must lay along the same radial line.
\item The ratio of the distances from $q$ and $r$ to the origin must equal the ratio between $\beta$ and $\alpha$.
\end{enumerate}
\end{lemma}

Lemma \ref{setB} is proved in Section \ref{lemmaProofs}. With these conditions in tow, we see that the pairs of points in $B$ are quite rare. Fix any radial line, $L$. Each point from $P \cap L$ can have its $\alpha$-line overlap with at most one $\beta$-line. Similarly, each point from $P \cap L$ can have its $\beta$-line overlap with at most one $\alpha$-line. So each point from $P \cap L$ can be in at most two pairs from $B$.
\begin{center}
\includegraphics[scale=1]{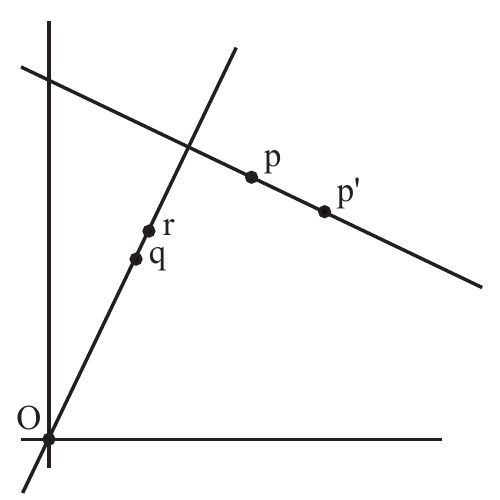}\\
\small{FIGURE 2: Here, the pair $(q,r)$ lies on a radial line. The points $p$ and $p'$ lie on the $\alpha$-line of $q$, which coincides with the $\beta$-line of $r$.}
\end{center}
Any pair, $(q,r)\in B,$ that lies on $L$ will contribute as many triples of points, $(p,q,r)$, to $\Pi_{\alpha,\beta}(P)$ as there are points coincident to both the $\alpha$-line of $q$ and the $\beta$-line of $r$. As these families of shared lines are parallel for point pairs along $L$, each point, $p$, can be on at most one $\alpha$-line, regardless of a possible overlap with a $\beta$-line. This means that each point $p\in P$ can be in at most one triple of the form $(p,q,r) \in \Pi_{\alpha,\beta}(P)$ with a pair of points, $q$ and $r$, from $L$. The total number of triples contributed by pairs of points in $L$ is therefore no more than $n$.

As there are no more than $n$ points, there can be no more than $n$ distinct radial lines to consider. Since each radial line can contribute no more than $n$ triples to $\Pi_{\alpha,\beta}(P)$, the maximum contribution to $\Pi_{\alpha,\beta}(P)$ by pairs in $B$ is no more than $n^2$. 
\end{proof}

\subsection{Proof of Theorem \ref{separation}}\label{sepProof}

\begin{proof} 

First, we define $\alpha$-lines and $\beta$-lines as in the proof of Theorem \ref{general}. Again, we let $\mathcal L_\alpha (p)$ and $\mathcal L_\beta (p)$ denote the set of all points incident to the $\alpha$-line or $\beta$-line for the given point $p$. Now, consider the $\alpha$-lines and $\beta$-lines for each $p \in P$. Referring back to the definition of $\Pi_{\alpha,\beta}(P)$, we see that a triple of points, $(p,q,r)$, will be in $\Pi_{\alpha,\beta}(P)$ precisely when $q$ lies on the $\alpha$-line of $p$ and $r$ lies on the $\beta$-line of $p$. So the quantity we aim to estimate is
\begin{equation}\label{bigEst}
|\Pi_{\alpha,\beta}(P)|=\sum_{p\in P} |\mathcal L_\alpha(p)||\mathcal L_\beta(p)|.
\end{equation}

It follows that we want a bound on the number of times that a point, $p \in P$, is incident to an $\alpha$-line or $\beta$-line. The following result will be proved in Section \ref{lemmaProofs}.

\begin{lemma}\label{STLemma}
In the setting above, the number of point-line incidences, $I$, is 
$$I \lesssim n^{\frac{4}{3}}.$$
\end {lemma}



Recalling the definitions of $\mathcal L_\alpha(p)$ and $\mathcal L_\beta(p),$ we write $I$ as:

$$
I=\sum_{p \in P} (|\mathcal L_\alpha(p)|+|\mathcal L_\beta(p)|)
$$

We now dyadically decompose $P$ into two families of disjoint sets defined by the number of incidences they contribute.

$$
P_j^{\alpha} := \{p \in P : 2^j \leq |\mathcal L_\alpha (p)| < 2^{j+1}\}
$$

Where $P_k^{\beta}$ is defined similarly. Now, if a point has roughly $2^j$ points from $P$ on its $\alpha$-line, and $2^k$ points from $P$ on its $\beta$-line, then it will be in the intersection of $P_j^{\alpha}$ and $P_k^{\beta}$. Let us define these intersections as:

$$P_{j,k} := P_j^{\alpha} \cap P_k^{\beta}.$$

The intersection of any $\alpha$-line or $\beta$-line with $[0,1]^2$ can be no longer than $\sqrt{2}$. All of our points are contained in $[0,1]^2$, so it will suffice to estimate the maximum number of points of $P$ on any line segment of length $\leq \sqrt{2}$. Fix such a segment, $\ell$. Recall the separation condition,
$$\min\{|p-q| : p,q\in P,~ p\neq q \}\geq \epsilon.$$
So each point on $\ell$ must have a vacant length of segment at least $\epsilon$ long in either direction. So if $\ell$ had points packed on it maximally, there would be no more than $$\frac{\sqrt{2}}{\epsilon}.$$
We can see that for every point $p \in P$,
\begin{equation}\label{sepLine}
|\mathcal L_\alpha(p)| \lesssim \epsilon^{-1}\text{ and }
|\mathcal L_\beta(p)| \lesssim \epsilon^{-1},
\end{equation}
so we can be assured that $P_{j,k}$ is empty for $j$ or $k$ bigger than $\left\lceil \log_2\left(\epsilon^{-1}\right)\right\rceil$. This also tells us that for all relevant indices $j$ and $k$ in the sums to follow, we have:
\begin{equation}\label{partC}
2^j,2^k \lesssim \epsilon^{-1}
\end{equation}

By combining the above:
\begin{align*}
n^\frac{4}{3} \gtrsim I &= \sum_{p \in P} (|\mathcal L_\alpha(p)|+|\mathcal L_\beta(p)|)\\&\approx \sum_{j=0}^{\left\lceil \log_2\left(\epsilon^{-1}\right)\right\rceil} \sum_{k=0}^{\left\lceil \log_2\left(\epsilon^{-1}\right)\right\rceil} \left(\sum_{p \in P_{j,k}} \left(2^j+2^k\right)\right)\\
&= \sum_{j=0}^{\left\lceil \log_2\left(\epsilon^{-1}\right)\right\rceil} \sum_{k=0}^{\left\lceil \log_2\left(\epsilon^{-1}\right)\right\rceil}\left( |P_{j,k}|\left(2^j+2^k\right)\right).
\end{align*}

So, for any pair of indices, $j$ and $k$, we have the following bound
\begin{equation}\label{partA}
|P_{j,k}|\left(2^j +2^k\right) \lesssim n^\frac{4}{3}.
\end{equation}

We now dyadically decompose the sum in \eqref{bigEst} as we did with the sum estimating $I$.

\begin{align}\nonumber
|\Pi_{\alpha,\beta}(P)|&=\sum_{p\in P} |\mathcal L_\alpha(p)||\mathcal L_\beta(p)|\\
&\approx  \sum_{j=0}^{\left\lceil \log_2\left(\epsilon^{-1}\right)\right\rceil} \sum_{k=0}^{\left\lceil \log_2\left(\epsilon^{-1}\right)\right\rceil}\left( |P_{j,k}|\left(2^j\right)( 2^k)\right)\label{bigSum}
\end{align}

Let $l$ and $m$ be a pair of indices that give the largest summand in \eqref{bigSum}. Now we have:

\begin{align}\nonumber
|\Pi_{\alpha,\beta}(P)|&\approx  \sum_{j=0}^{\left\lceil \log_2\left(\epsilon^{-1}\right)\right\rceil} \sum_{k=0}^{\left\lceil \log_2\left(\epsilon^{-1}\right)\right\rceil}\left( |P_{j,k}|\left(2^j\right)( 2^k)\right)\\\nonumber
&\lesssim   \sum_{j=0}^{\left\lceil \log_2\left(\epsilon^{-1}\right)\right\rceil} \sum_{k=0}^{\left\lceil \log_2\left(\epsilon^{-1}\right)\right\rceil}\left( |P_{l,m}|(2^l)\left( 2^m\right)\right)\\
&\lesssim   \left\lceil \log_2\left(\epsilon^{-1}\right)\right\rceil \left\lceil \log_2\left(\epsilon^{-1}\right)\right\rceil \left( |P_{l,m}|(2^l)\left( 2^m\right)\right)\label{biggerSum}
\end{align}

Burying the constants from the logarithms in \eqref{biggerSum}, we get:
\begin{equation}\label{partB}
|\Pi_{\alpha,\beta}(P)| \lesssim \left( |P_{l,m}|(2^l)\left( 2^m\right)\right) \log \epsilon^{-1}.
\end{equation}

 Finally, by \eqref{partB}, adding in a term of $2^m$, \eqref{partA}, and \eqref{partC}, we conclude

\begin{align*}
|\Pi_{\alpha,\beta}(P)| &\lesssim \left( |P_{l,m}|(2^l)\left( 2^m\right)\right)\log \left(\epsilon^{-1}\right)\\ &\lesssim \left( |P_{l,m}|\left(2^l+2^m\right)\left( 2^m\right)\right)\log \left(\epsilon^{-1}\right)\\
&\lesssim  n^\frac{4}{3}\left( 2^m\right)\log\left(\epsilon^{-1}\right)\\
&\lesssim n^\frac{4}{3}\epsilon^{-1}\log \left(\epsilon^{-1}\right),
\end{align*}
as desired.
%
%
%
%
\end{proof}

\section{Proofs of Lemmas}\label{lemmaProofs}

\subsection{Proof of Lemma \ref{setB}}

To see $(1)$, notice that any line that yields the prescribed dot products for a given point is perpendicular to that point's radial line. For the $\alpha$-line of $q$ and $\beta$-line of $r$ to be coincidental, they must both be perpendicular to the same radial line, and thus, be generated from two points upon the same radial line.

For $(2)$, assuming the $\alpha$-line of $q$ is coincidental to the $\beta$-line of $r$, we get the following two equations of lines in the plane:
$$y=\frac{\alpha}{q_2} - \frac{q_1}{q_2}x$$
$$y=\frac{\beta}{r_2} - \frac{r_1}{r_2}x.$$
We know that $q$ and $r$ are on the same radial line, so there must exist a $\lambda>0$ such that,
$$q_1 = \lambda r_1, \text{ and } q_2=\lambda r_2.$$
We set the equations equal to one another and get:
$$\frac{\alpha}{q_2} - \frac{q_1}{q_2}x=\frac{\beta}{r_2} - \frac{r_1}{r_2}x.$$
By substituting for the coordinates of $q$,
$$\frac{\lambda \alpha}{r_2} - \frac{\lambda r_1}{\lambda r_2}x=\frac{\beta}{r_2} - \frac{r_1}{r_2}x.$$
After simplifying, we see that
$$\lambda \alpha = \beta.$$
\subsection{Proof of Lemma \ref{STLemma}}

There is a small technical obstruction to a direct application of the celebrated Szemer\'edi-Trotter point-line incidence theorem, which is that two types of line may be coincident. This turns out to not be a problem in our case. By appealing to the definitions of $\alpha$-line and $\beta$-line, we can see that no two points can determine the same $\alpha$-line or $\beta$-line. However, as we have seen above, it is possible for the $\alpha$-line of a point to overlap the $\beta$-line of a point (a different point, unless $\alpha = \beta$).

We prove Lemma \ref{STLemma} using techniques of Sz\' ekely, from \cite{Szek}. Draw a graph (possibly a multigraph) with the $n$ points as vertices, and the segments of the $\alpha$-lines and $\beta$-lines connecting adjacent points as edges. Observe that no pair of points can have more than two edges connecting them, and this can only happen if the $\alpha$-line of a point is coincident to the $\beta$-line of a point. If this does happen, we can replace the line segments by two curves, whose endpoints are the adjacent points in question, drawn in such a way so as not to cross any more or fewer edges than the initial line segment crossed.

Notice that the number of edges contributed by each line is equal to the number of point-line incidences on that line minus one (For example, if there are seven points on a given line, we would connect consecutive points with six edges.). So the number of edges is equal to the number of incidences minus the number of lines that contribute incidences. There are exactly $2n$ lines ($n$ of each type) that may contribute incidences. Let $I$ denote the number of point-line incidences, and $e$ denote the number of edges in the graph. We have that
\begin{equation}\label{edges}
e\geq I-2n.
\end{equation}

We let $cr(G)$ denote the {\it crossing number of $G$}, which is the maximum number times that edges of $G$ must cross one another at some point which is not a vertex, for any redrawing of $G$. We appeal to the crossing number lemma in \cite{Szek}.

\begin{lemma}\label{szek}
Given a topological multi-graph, $G$, with $v$ vertices, $e$ edges, and a maximum edge multiplicity of $m$, if $e>5mv$, then
$$cr(G) \gtrsim \frac{e^3}{mv^2}.$$
\end{lemma}

For our setup, we have that $m\leq 2$, as there can be no more than two edges between any given pair of points. Now, either $e\leq5mv$ or $e>5mv$.
In the first case, we recall \eqref{edges} to see that
$$I\leq e+2n \leq 5mv+2n \leq 10n+2n=12n.$$
In the case that $e>5mv$, we appeal to Lemma \ref{szek} and get
$$\frac{(I-2n)^3}{2n^2}\leq \frac{e^3}{mv^2} \lesssim cr(G).$$
Notice that the crossing number of the graph can be no more than the number of times that the $\alpha$-lines and $\beta$-lines crossed one another. Since there are $2n$ total lines, and each line could potentially cross almost all of the others, the total number of crossing lines is bounded above by $(2n)^2$. As this would correspond to a drawing of the graph, $G$, we are guaranteed that the crossing number of $G$ is no more than $(2n)^2$. Comparing upper and lower bounds on $cr(G)$ yields
$$\frac{(I-2n)^3}{2n^2}\lesssim n^2,$$
which gives us that
$$I\lesssim n^\frac{4}{3}.$$
In either case, the claimed estimate holds.

\vskip.5in

\end{document}